\documentclass{amsart}
\usepackage{extarrows}
\usepackage{tikz-cd}
\usepackage{amsmath,amssymb,amsfonts,amscd}
\usepackage{amsthm}
\usepackage{hyperref}

\usetikzlibrary{matrix}
\usepackage{url}

\usepackage{ upgreek }

\usepackage[shortlabels]{enumitem}

\newtheorem{lettertheorem}{Theorem}

\newtheorem*{samelettertheoremA}{ Theorem A}
\newtheorem*{samelettertheoremB}{ Theorem B}

\newtheorem{theorem}{Theorem}[section]
\newtheorem*{theorem*}{Theorem}

 \newtheorem{corollary}[theorem]{Corollary}
 \newtheorem{proposition}[theorem]{Propsosition}
 \newtheorem*{mainthm*}{Main Theorem}

\theoremstyle{definition}
\newtheorem{definition}[theorem]{Definition}
\newtheorem{example}[theorem]{Example}

\theoremstyle{remark}
\newtheorem{remark}[theorem]{Remark}

\definecolor{alert}{rgb}{0.8,0,0.3}

\newcommand{\alert}[1]{%
	\marginpar{%
		\ifodd\value{page} \raggedright \else \raggedleft \fi
		\footnotesize{\textcolor{alert}{#1}}
	}
}

\newcommand{\Ric}{{\rm Ric}}
\newcommand{\tr}{{\rm tr}}

\newcommand{\V}{{\rm V}}

\begin{document}

\title[Torsional Rigidity evolution under geometric flows]{ Evolution of the Torsional Rigidity under Geometric Flows}

\author{Vicent Gimeno i Garcia}
\address{Department of Mathematics, Universitat Jaume I-IMAC,   E-12071, 
Castell\'{o}, Spain}
\email{gimenov@mat.uji.es}

\author{Fernán Gonz\' alez-Ib\' a\~nez}
\address{Department of Mathematics, Universitat Jaume I,   E-12071, 
Castell\'{o}, Spain}
\email{fgonzalez@uni-muenster.de}

\thanks{Research partially supported by the Research grant  PID2020-115930GA-100 funded by MCIN/ AEI /10.13039/50110001103 and  by the project AICO/2023/035 funded by Conselleria
d’Educació, Cultura, Universitats i Ocupació}




\dedicatory{}

\begin{abstract}This paper explores the behavior of the torsional rigidity of a precompact domain as the ambient manifold evolves under a geometric flow. Specifically, we derive bounds on torsional rigidity under the Ricci Flow for Heisenberg spaces and homogeneous spheres. Additionally, we establish bounds under the Inverse Mean Curvature Flow for strictly convex, free-boundary, disk-type hypersurfaces within a ball. In this latter case, by extending the analysis to the maximal existence time of the flow, we obtain inequalities of comparison  with the flat disk for both volume and torsional rigidity.
\end{abstract}

\maketitle
\tableofcontents

\section{Introduction}\label{sec:intro}
For a given domain $\Omega\subset M^n$ in a $n$-dimensional Riemannian manifold $(M^n,g)$ we can define the torsional rigidity of this domain as the integral of the solution to the following Poisson problem with Dirichlet boundary condition, 
\begin{equation}\label{eq:unua}
    \left\lbrace
\begin{array}{rlcc}
\Delta_g E&=-1& {\rm in }& \Omega,\\
    {E}&=0&\text{ \rm in } &\partial\Omega.
\end{array}
\right.
\end{equation}
This problem has been widely studied in the last years, see for instance  \cite{mcdonald}, \cite{Grigoryan},\cite{MARKVORSEN_PALMER_2006} or \cite{erick}. 
For a deeper discussion of the meaning of the integral the reader is referred to section \ref{sec:torormean}. 

Since the essence of the problem is fundamentally related to the Riemannian metric, a natural question arises when studying this topic is what happens when changes  are made  to the metric tensor. There exist different ways to perform changes on the metric, for example conformal changes or geometric flows. In this article the spotlight will be  in geometric flow as, for instance,  Ricci and Inverse Mean Curvature flow. For a more detailed explanation we refer to section \ref{sec:GF}.

A  solution of a geometric flow is a continuous path in the space of metrics $g:[0,T_{\rm max})\to \mathcal{M}(M^n)$ that  follows a geometric equation, $\frac{\partial g}{\partial t}= F(t)$ satisfied up to a maximal time of existence of the flow denoted as $T_{\rm max}$. Thus for each $t\in [0,T_{max})$ the solution to the previously mentioned Poisson problem will be different. Hence we will obtain  a one-parameter family of solutions, $\{ E_t\}_{t\in [0,T_{max})}$ to the problem. The torsional rigidity of $\Omega$ at time $t$ will be obtained therefore by  the following integral
\begin{equation*}
   \mathcal{T}(\Omega_t):=\int_\Omega E_t\, d{\rm V}_t  
\end{equation*}
where $d{\rm V}_t$ is the volume element associated to $g_t$ and $E_t$ is the solution of \eqref{eq:unua} at time $t$. We can therefore  consider the torsional rigidity for a given domain as a function that depends on the time. 

Geometric flows determines the evolution of various properties associated with the Riemannian metric. In this paper, the controlled evolution of these quantities under the flow is leveraged to derive bounds on the torsional rigidity, utilizing the variational characterizations introduced in Section \ref{sec:bounds}.
In the particular case of the Ricci flow, focusing in manifolds with constant scalar curvature and bounded Ricci tensor we have obtained several bounds for the torsional rigidity. The result is summarized in the following theorem:

\begin{lettertheorem} \label{Thm:Ric}Let $(M^n,(g_t)_{t\in [0,T_{\rm max})})$ be a Ricci flow solution. Suppose that for any $t\in [0,T_{\rm max})$ and any $p\in M$
$$
 {\rm Scal}_{g_t}(p)=b(t).
$$
Suppose moreover that for any $t\in [0,T_{\rm max})$, any $p\in M$ and any $v\in T_pM$, there exists two functions $A,B:\mathbb{R}\to\mathbb{R}$ such that 
$$
A(t)g_t(v,v)\leq {\rm Ric}_{g_t}(v,v)\leq B(t)g_t(v,v). 
$$
Then for any precompact domain $\Omega\subset M$ with smooth boundary 
$$
\begin{aligned}
 e^{\int_0^t\left(-b(s)-2B(s)\right)ds}\mathcal{T}(\Omega_0)\leq\mathcal{T}(\Omega_t)\leq& e^{\int_0^t\left(-2A(s)-b(s)\right)ds}\mathcal{T}(\Omega_0). 
\end{aligned}
$$
Which implies
$$
\begin{aligned}
 e^{-2\int_0^tB(s)ds}\frac{\mathcal{T}(\Omega_0)}{\V(\Omega_0)}\leq\frac{\mathcal{T}(\Omega_t)}{\V(\Omega_t)}\leq& e^{-2\int_0^tA(s)ds}\frac{\mathcal{T}(\Omega_0)}{\V(\Omega_0)}. 
\end{aligned}
$$
Where here  ${\rm V}(\Omega_t)$ denotes the volume of $\Omega$ with respect to $g_t$.
\end{lettertheorem}
To check the consistence of the above Theorem we can apply it to the case of Einstein manifolds $(M,g_0)$ with initial metric tensor satisfying $\Ric_{g_0}=\lambda g_0$. In this case applying the above theorem, see section \ref{sec:Einstein}, we obtain for any precomapct domain $\Omega\subset M$ with smooth boundary that
$$
    \mathcal{T}(\Omega_t)= \left(1-2\lambda t\right)^{\frac{n+2}{2}}\mathcal{T}(\Omega_0)  
$$
as corresponds to the evolutions of the metric tensor by  the form $g_t= (1-2\lambda t)g_0$.

 By applying Theorem \ref{Thm:Ric}, for manifolds whose Ricci flow satisfies that 
$
 {\rm Scal}_{g_t}(p)=b(t)
$, and the Ricci tensor is negative (or positive) definite along the flow we obtain, see corollary \ref{cor:Ric}, we obtain that the function 
$$
t\mapsto \frac{\mathcal{T}(\Omega_t)}{\V(\Omega_t)}
$$
is a non decreasing (non increasing for positive definite) function for $t\in [0,T_0]$, and for any precompact domain $\Omega$ with smooth boundary the function
$$
t\mapsto \frac{\mathcal{T}(\Omega_t)}{\V^3(\Omega_t)}
$$
is a non increasing (non decreasing respectively) function for $t\in [0,T_0]$.

 Theorem \ref{Thm:Ric} is  applied as well in this article  for the case of homogeneous spaces because since the Ricci flow preserves isometries, the manifold remains a homogeneous manifold along the flow, and the hypothesis on the scalar curvature of Theorem \ref{Thm:Ric} holds. In this case we focus on the Heisenberg group and on homogeneous spheres. In the case of the Heisenberg group ${\rm Nil}_3$, see section \ref{Heisenberg}, we obtain the for any precompact domain $\Omega\subset{\rm Nil}_3$ with smooth bounday, the function
 $$
t\mapsto \V(\Omega_t)\mathcal{T}(\Omega_t),
    $$
    is non decreasing on $t\in [0,\infty)$, and the function
    $$
t\mapsto \frac{\mathcal{T}(\Omega_t)}{\V(\Omega_t)^3},
    $$
    is non increasing on $t\in [0,\infty)$.
        
For the homogeneous spheres, see section \ref{sec:su2}, with initial parameters
    $
\epsilon A_0\leq C_0\leq B_0$ such that $ \delta:=\frac{B_0-\epsilon A_0}{\epsilon A_0}<1
$.
It has been computed that,
$$
\left[1-\frac{4(1+6\delta)t}{B_0}\right]^{\frac{5(1+\delta)^3}{2(1+6\delta)}}\mathcal{T}(\Omega_0)\leq \mathcal{T}(\Omega_t)\leq \left[1-\frac{4t}{B_0}\right]^{\frac{5(1-\delta)}{2(1+\delta)}}\mathcal{T}(\Omega_0),
$$
 for any $t<\min\left\{\frac{B_0}{4(1+6\delta)},T_{\rm max}\right\}$.
In the case of the inverse mean curvature flow, bounds for the torsional rigidity have been obtained for domains that satisfy the following conditions: they are totally bounded (i.e., precompact) and have a strictly convex image, $H > 0$. The time for which the hypersurface remains strictly convex, denoted by $T_{mc}$, has been studied by Lambert and Scheuer for strictly convex, free boundary disk-type hypersurfaces in the ball in \cite{Lambert2016inver-34677} or by Urbas in \cite{Urbas}, who considered a more general case of hypersurface evolution in $\mathbb{R}^{n+1}$. In the case of non-convex hypersurfaces we highlight the work of by Gerhardt in \cite{maxtime}. The bounds for the torsional rigidity in precompact domains in strictly convex hypersurfaces can be summarized in the following theorem, see section \ref{sec:IMCF},

\begin{lettertheorem}\label{teo:unua}
Let M be an n-dimensional manifold and let $\Omega\subset M$ be an open and precompact domain of $M$ with smooth boundary $\partial \Omega\neq \emptyset$. Then let
    $\varphi:M\times[0,T_{\rm mc})\to \mathbb{R}^{n+1}$ by an inverse mean curvature flow from $M$ to $\mathbb{R}^{n+1}$ such that $\varphi_t(M)=\Sigma$ is strictly  convex hypersurface of $\mathbb{R}^{n+1}$.Thence the function
    $$
 t\mapsto ({\rm V}(\Omega_t))^{-3}\cdot\mathcal{T}(\Omega_t)
    $$
    is non-increasing for $t\in [0,T_{\rm mc})$, and the function
    $$
 t\mapsto ({\rm V}(\Omega_t))^{-1}\cdot\mathcal{T}(\Omega_t)
    $$
    is non-decreasing for $t\in [0,T_{\rm mc})$. Where here  ${\rm V}(\Omega_t)$ denotes the volume of $\Omega$ with respect to $g_t$.
\end{lettertheorem}
In the case $\Sigma$ be a strictly convex free boundary disk-type hypersurface  in the ball $\mathbb{B}\subset \mathbb{R}^{n+1}$, as in definition \ref{def:fb}, inspired in the work of Pak Tung Ho and Juncheol Pyo, in \cite{Ho-Pyo}, taking into account that $\Sigma$ converges to the flat disk when $t$ tends to $T_{\rm max}$, it has been obtained bounds for the torsional rigidity for this family of hypersurfaces in comparison with the flat disk as follows,
$$
\frac{\mathcal{T}(\Sigma)}{\V^3(\Sigma)}\geq \frac{\mathcal{T}(\mathbb{D})}{\V^3(\mathbb{D})},\quad \frac{\mathcal{T}(\Sigma)}{\V(\Sigma)}\leq \frac{\mathcal{T}(\mathbb{D})}{\V(\mathbb{D})},
$$
where $\V(\mathbb{D})$ and $\mathcal{T}(\mathbb{D})$ are the volume and the torsional rigidity of the unit disk in $\mathbb{R}^n$, and $\V(\Sigma)$ denotes the volume of the hypersurface.   Then
    $$
\V(\Sigma)\leq \V(\mathbb{D}),\quad   {\rm and}\quad \mathcal{T}(\Sigma)\leq \mathcal{T}(\mathbb{D}).
    $$
We must remark here that the assumption that the mean curvature is strictly positive $H>0$ is crucial for deriving our results. In contrast, for minimal ($H=0$) hypersurfaces $P\subset \mathbb{R}^{n+1}$ with $\vec{0}\in P$, by applying the results established by Markvorsen and Palmer in \cite{MARKVORSEN_PALMER_2006} (see remark 2.2 in their paper), it can be concluded for the region $\Sigma=P\cap \mathbb{B}$ contained within $\mathbb{B}$ that:
    $$
\V(\Sigma)\geq \V(\mathbb{D}),\quad {\rm and}\quad    
\mathcal{T}(\Sigma)\geq \mathcal{T}(\mathbb{D}).
    $$
Thus, our results for hypersurfaces with strictly positive mean curvature can be interpreted as the inverse of those obtained by Markvorsen and Palmer concerning minimal hypersurfaces.

\subsection{Plan of the paper }
 The structure of the paper is as follows: In section \ref{sec:torormean} the relationship between the torsional rigidity and the mean exit time is discussed.
In section \ref{sec:bounds} we will introduce two variational approaches for the torsional rigidity which will be used along the paper to obtain  upper and lower bounds for the torsional rigidity. In section \ref{sec:GF} geometric flows and related inequalities for the evolution of the trosional rigidity will be introduced. As a particular case, in section \ref{sec:cinc},  Theorems \ref{Thm:Ric} and \ref{teo:unua} will be proved for the case of Rici flow and Inverse Mean Curvature flow respectively. Finally, in section \ref{sec:examples} we will present some examples of application.

\section{Poisson problem: two names one function. }\label{sec:torormean} 
In classical literature, the solution to the problem \eqref{eq:unua} can be approached from two perspectives: either as an auxiliary function whose integral yields the torsional function, or as the mean exit time function for the Brownian motion. In the following paragraphs, both concepts will be introduced, and the Poisson problem will be derived from each perspective.

The first point of view is related to the theory of elasticity. Where the integral of the solution from the afforded mentioned problem give us the torque required when twisting an elastic beam with the domain of study as  uniform cross section. The second one is related to Brownian motion on domains. The solution of the problem \eqref{eq:unua} can be seen as the expected time for a Brownian particle starting at a point $x\in \Omega$ to reach the boundary of the domain.

\subsection{The Poisson problem as torsional rigidity }
Following the idea of Landau and Lipschitz in \cite{LL} and the articles of Polya \textit{et. al.}  \cite{PolyaSV}, \cite{Plya1948TorsionalRP} and \cite{Polya_Variational}, a beam of length \(L \in (0, \infty)\) can be modeled as \(\mathcal{B} = \Omega \times [0, L] \subset \mathbb{R}^3\). Where \(\Omega \subset \mathbb{R}^2\) is a bounded, simply connected and open set representing the cross-section.
  In this model, a coordinate system where the \(z\)-axis corresponds to the beam's length, and the \(x\)- and \(y\)-axes represent the coordinates within \(\Omega\) can be defined.
In the twisting,  the beam remains straight, but each cross-sectional slice (i.e., each copy of \(\Omega\)) undergoes a rotation of \(\theta\) radians per unit length.

The stress energy tensor, see \cite{LL},  can be described in terms of the deformation tensor as follows,
$$
\sigma_{xx}=\sigma_{yy}=\sigma_{zz}=0=\sigma_{xy},\quad \sigma_{xz}= 2\mu \frac{\theta}{2}\bigg( \frac{\partial \Psi}{\partial x} -y \bigg)\quad {\rm and} \quad \sigma_{yz}= 2\mu \frac{\theta}{2}\bigg( \frac{\partial \Psi}{\partial y} +x \bigg).
$$
Where $\Psi(x,y)$ is a function that gives the vertical displacement under the twist and $\mu$ is the modulus of the transversal elasticity that here and so on we will take it as $\mu=1$. At this point it is convenient to utilize an auxiliary function $\chi(x,y)$ defined as follows, 
$$
\sigma_{xz}= 2\mu\theta \frac{\partial \chi}{\partial y} \quad and \quad \sigma_{yz}= -2\mu\theta \frac{\partial \chi}{\partial x}.
$$
Thence, $
\frac{\partial \Psi}{\partial x}= y+2\frac{\partial \chi}{\partial x}\quad and \quad
\frac{\partial \Psi}{\partial y}= -x-2\frac{\partial \chi}{\partial y}
$. Finally, taking the derivatives $\frac{\partial^2\chi}{\partial x^2}$, $\frac{\partial^2\chi}{\partial y^2}$ and assuming that in the boundary the  forces of torsion are much smaller than the interior of the domain we obtain 
\begin{equation*}
    \left\lbrace
\begin{aligned}
    \Delta \chi=&-1 \text{ in } \Omega,\\
    {\chi}=&0\text{ in } \partial\Omega.
\end{aligned}
\right.
\end{equation*}
Which implies that the total energy of the beam will be
 $$
E(\mathcal{B})= \int^L_0 \int_\Omega (\nabla \chi)^2 4 \theta^2 d\mu_\Omega dz,
 $$
and the resulted torque due to the stress function $\mathcal{T}(\Omega)$, \emph{i.e.},  the  torsional rigidity, is proportional to
$$
\int_\Omega \chi d\mu_\Omega.
$$

\subsection{The Poisson problem as mean exit time} 
In this subsection we will follow the construction of a Brownian motion modeled by a diffusion process as Alexander Grigor’yan and  Telcs did in \cite{gg2} and then the mean exit time function will be expressed as the solution of Poison problem \eqref{eq:unua} as was stated by Dynkin in \cite{Dynkin}.  

Let $(M,g)$ be a connected Riemannian manifold equipped with the volume form $d\V_g$.  
The associated Dirichlet form will be used to construct a diffusion process. This  is a way to model Markov process. For a better understanding of the Dirichlet form and the associated Markov process we refer to \cite{FukushimaOshimaTakeda+2010}.

Recall that in a Riemannian manifold the set of functions $H^1(M)$  given by 
$$H^1(M):=\left\{ f\in L^2(M)\, :\, \Vert\nabla f\Vert\in L^2(M^n)\right\}$$ 
is a dense subset of $L^2(M)$ and a Dirichlet form can be defined as the fowolling bilinear map  
\begin{equation*}
\begin{aligned}
\mathcal{E}:H^1(M)\times H^1(M) \to L^2(M),\quad    (f,h)\mapsto& \mathcal{E}(f,h):=\int_M g(\nabla f, \nabla h)\ d\V_g.
\end{aligned}
\end{equation*}
This bilinear map is nonnegative definite, closed, satisfies the Markovian property, and can be generated using the Laplace-Beltrami operator $\Delta_g$ in $L^2(M)$. Indeed, the   domain ${\rm dom}(\Delta_g)$ is precisely the dense subset of $H^1(M)$ given by
$$
{\rm dom}(\Delta_g)=\left\{f\in H^1(M)\,:\,\mathcal{E}(f,h)=\int_M h\Delta_g f\ d\V_g,\quad \forall h\in H^1(M)\right\}.$$

Every Dirichlet form, see Theorem 7.2.1 of \cite{FukushimaOshimaTakeda+2010}, has  associated a diffusion process $\{\{X_t\}_{t\geq 0},\{\mathbb{P}_x\}_{x\in M }\}$  with the initial point $x\in M$. This  diffusion process is associated with the heat semi-group, in general known as transition semi-group, $P:= \{e^{t\Delta_g}\}_{t\geq 0}$\footnote{There exist two conventions at the moment of defining the heat semi-group. The first one is using that $P$ is generated by the usual Laplace-Beltrami operator and the second one is that $P$ is generated by $\frac{\Delta_g}{2}.$ This decision will have consequences at the moment of setting up the Poisson problem. In this article the first one is chosen. } as follows: for any $f\in L^2(M)$ and $x\in M$,
$$ \mathbb{E}_xf(X_t) = P_t(f(x)).$$
The transformation $P_t(f)(x)$ of the function $f$  can be obtained as a convolution of the heat kernel $p_t$ in an integral form as 
$$P_t(f)(x)= \int_M p_t(x,y)f(y)d\V_g(y).$$
 
This Markovian process, derived form the Dirichlet form, is known as Brownian motion.
\begin{definition}Let $(M,g)$ be a Riemannian manifold.
The Brownian motion on $(M,g)$ is the unique Markov process $\{X_t\}$ with state space $M$ with measure $d\V_g$ and transition densities $p_t(x,\Omega)= \int_{\Omega} p_t(x,y)d\V_g(y)$.
\end{definition}
A natural question that arises when studying Brownian motion in a domain $\Omega$ concerns the first time $\tau_p$ when the Brownian motion will leave the  domain $\Omega$, starting at a point $p \in \Omega$. To estimate this first exit time, one can consider the diffusion process terminated, ``killed'', when it  reaches the boundary of $\Omega$. 

The relationship between the elements of the transition semigroup and the expected value of the first exit time is given by:
$$
P^\Omega_t(f)(x) = \mathbb{E}_x (\mathbf{1}_{t < \tau_\Omega} f(X_t)),\quad\mathbf{1}_{t < \tau_\Omega}:=\left\{\begin{array}{cc}
  1   & t < \tau_\Omega,\\
  0   & t\geq \tau_\Omega.
\end{array}\right.
$$
Roughly speaking, $\mathbf{1}_{t < \tau_\Omega}$ is forcing  the Brownian particle to remain within $\Omega$ up to time $t$. This can also be expressed as a convolution with a heat kernel:
$$
P^\Omega_t(f)(x) = \int_\Omega f(y) p_t^\Omega(x, y) \, d\mathrm{V}_g,
$$
but now  $p_t^\Omega$ is the heat kernel of $\Omega$ with Dirichlet boundary condition. For any measurable function in $\Omega$, by using the Green operator in $\Omega$ we have that, 
$$
\mathbb{E}_x\left(\int_{0}^{\tau_\Omega} f(X_t)dt\right)=\int_{0}^\infty P^\Omega_t(f)(x) \, dt=:G^\Omega(f)(x).
$$
In particular setting $f(x)=1$ we obtain the mean exit time function $\mathbb{E}_x(\tau_\Omega)$ as
\begin{equation}
   \mathbb{E}_x(\tau_\Omega)=G^\Omega(1)(x).
\end{equation}
Therefore, see for instance  \cite{Bessa2020} for the properties of the Green operator, we obtain the following characterization of the mean exit time function $\mathbb{E}_x(\tau_\Omega)$  as  the desired Poisson problem with Dirichlet boundary condition 
\begin{equation*}
    \left\lbrace
\begin{aligned}
    \Delta  \mathbb{E}_x(\tau_\Omega) =&-1 \text{ in } \Omega,\\
    { \mathbb{E}_x(\tau_\Omega)}=&0\text{ in } \partial\Omega.
\end{aligned}
\right.
\end{equation*}

\begin{example}\label{ex1}
    Let $(\mathbb{R}^n,g_\mathbb{E})$ with $n\geq 2$ be the Euclidean space with the canonical  Euclidean metric. Let the domain of study $\Omega = D_{1}(\vec{0})$  be the disk of radius $1$ centered at $\vec{0}\in \mathbb{R}^n$. The mean exit time for this domain will be 
    \begin{equation}\label{eq:tria}
 \mathbb{E}_x(\tau_\Omega)=  \frac{1}{2n}\left(1-\Vert x\Vert ^2\right).      
    \end{equation}

 For the case $n>1$ this can be proved by using the Green function of the disk $g^{\Omega}(x,y)=\int_0^\infty p_t^\Omega(x,y)dt$ and taking into account the the problem is rotationally symmetric, see \cite{Bessa2020} for radial Green functions in rotationaly symmetric balls, as  follows
  \begin{equation*}
    \begin{aligned}
        \mathbb{E}_{x}(\tau_\Omega)=& G^\Omega(1)= \int_{D_{1}(0)} g^\Omega(x,y) d\V_{g_E}(y)\\
        =& \int_{\Vert x\Vert}^1 \frac{t}{n}dt = \frac{1}{2n}\left(1-\Vert x\Vert^2\right).
    \end{aligned}
    \end{equation*}
Observe that \eqref{eq:tria} can be proved in an alternative way by taking into account that
    $$
\Delta \mathbb{E}_x(\tau_\Omega)=\left(\frac{\partial^2}{\partial x_1^2}+\cdots+\frac{\partial^2}{\partial x_n^2}\right)\frac{1-x_1^2-\cdots-x_n^2}{2n}=-1,
    $$
and $\mathbb{E}_x(\tau_\Omega)=0$  for   $x\in \partial \Omega$.
  \end{example}

\section{Two variational approaches for the torsional rigidity}\label{sec:bounds}

In order to study the evolution of the torsional rigidity of a domain within a Riemannian manifold as the metric tensor evolves under a geometric flow, it will became necessary to characterize the torsional rigidity through two variational approaches. Specifically, in Proposition \ref{prop:upper_bounds_torsion}, we will express the torsional rigidity as the supremum of a functional defined over certain functions within the precompact domain. Furthermore, in Proposition \ref{prop:lower_bounds_torsion}, we will represent it as the infimum of the energy associated with specific vector fields defined on the domain.

The first variational approach is a generalization of a variational approach made by Polya in the paper \cite{Polya_Variational}. By making use that the torsional rigidity problem arise from a Poisson problem, we can rewrite it as the critical point of an energy functional in the space of smooth functions $u\in C^\infty(\Omega)$ with $u\big|_{\partial \Omega}=0$. The variational characterization of the torsional rigidity introduced by Polya for domains in the $2$-dimensional Euclidean space $\mathbb{R}^2$ is
\begin{equation}\label{eq:torsionalvariation}
     \mathcal{T}(\Omega)= \sup \left\{ \frac{\left(\int_\Omega  f\, { d} {\rm A}\right)^2}{\int_\Omega \Vert \nabla f\Vert^2\, { d}{\rm A}} \, :\, f\in C^\infty(\Omega),\, f|_{\partial\Omega = 0}\right\}\cdot
\end{equation}
The main idea of the proof is that assuming  $f\in C^\infty(\Omega)$ and $f\big|_{\partial \Omega}=0$, the integral of $f$ can be bounded in terms of  the Dirichlet form because  
\begin{equation*}
    \begin{aligned}
        \mathcal{E}(f, \chi) & =\iint_\Omega\left(f_x \chi_x+f_y \chi_y\right) d A \\
        & =\iint_\Omega\left[\frac{\partial\left(f \chi_x\right)}{\partial x}+\frac{\partial\left(f \chi_y\right)}{\partial y}+2 f\right] dA \\
        & = \iint_\Omega f dA+\int_{\partial \Omega} f \frac{\partial \chi}{\partial n} dL \\
        & = \iint_\Omega f d A,
        \end{aligned}
\end{equation*}
where $\chi$ is the auxiliary function defined in the section 2.1. This implies that
$$
\left(\int_\Omega  f\, { d}{\rm A}\right)^2=\left(\mathcal{E}(f,\chi)\right)^2\leq \mathcal{E}(f,f)\cdot \mathcal{E}(\chi,\chi),
$$
with equality obtained when $f(x,y) = \chi(x,y)$. The variational characterization of Polya \eqref{eq:torsionalvariation} is generalized in  following proposition  for precompact domains of Riemannian Manifolds.

\begin{proposition}\label{prop:upper_bounds_torsion}Let $(M,g)$ be a Riemannian manifold.  Let $\Omega\subset M$ be an open precompact domain.    Then,
    
    \begin{equation*}
        \mathcal{T}(\Omega) = \sup\left\lbrace  \frac{\left(\int_\Omega  u\, {d}\V\right)^2}{\int_\Omega \Vert \nabla u\Vert^2\, { d}\V} \, :\,  u\in C^{\infty}_0(\Omega)\, {\rm and}\, 
         \int_\Omega \Vert \nabla u\Vert\, { d}\V\neq 0
    \right\rbrace.
    \end{equation*}
\end{proposition}
\begin{proof}
Let $E$ be the mean exit time function in $\Omega$, \emph{i.e} the solution of the Poisson problem \eqref{eq:unua},  and let $u\in \{ f\in C^{\infty}_0(\Omega) : \int_\Omega \Vert \nabla f\Vert\, {d}\V\neq 0\}$. Then,
    \begin{equation}\label{eq:3nov}
    \begin{aligned}
        \int _\Omega u\, { d}\V=&-\int_\Omega u\Delta E\, { d}\V=-\int_\Omega{\rm div}_g\left(u\nabla E\right)\, {d}\V+\int_\Omega\langle\nabla u,\nabla E\rangle\, { d}\V\\
        =&\int_\Omega\langle\nabla u,\nabla E\rangle\, { d}\V\leq \int_\Omega\Vert\nabla u\Vert \, \Vert \nabla E\Vert\,  {d}\V. 
    \end{aligned}
    \end{equation}
    Hence, by applying the Cauchy-Schwarz inequality 
    $$
\int_\Omega u\, {d}\V\leq \left(\int_\Omega\Vert\nabla u\Vert^2 \,  {d}\V\right)^{\frac{1}{2}}\left(\int_\Omega \Vert \nabla E\Vert^2\,  {d}\V\right)^{\frac{1}{2}}.
    $$
    Since
    $$
\int_\Omega \Vert \nabla E\Vert^2\, d{\rm V} =\int_\Omega {\rm div}_g(E\nabla E)\, d{\rm V}-\int_\Omega E\Delta E\, d{\rm V}=\mathcal{T}(\Omega),
    $$
    and we are assuming that $\int_\Omega\Vert\nabla u\Vert^2 \,  { d}\V\neq 0$, we obtain that 
    $$
\frac{\left(\int_\Omega  u\, {d}\V\right)^2}{\int_\Omega \Vert \nabla u\Vert^2\, {d}\V}\leq \mathcal{T}(\Omega)=\frac{\left(\int_\Omega  E\, {d}\V\right)^2}{\int_\Omega \Vert \nabla E\Vert^2\, { d}\V},
    $$
    and the proposition is proved.
 \end{proof}
 
The characterization of the torsional rigidity as an infimum can be stated in the following proposition

\begin{proposition}\label{prop:lower_bounds_torsion}
    Let $(M,g)$ be a Riemannian manifold, let $\Omega\subset M$ an open subset of $M$ with compact closure and smooth boundary $\partial\Omega\neq \emptyset$.    Then,
    $$
\mathcal{T}(\Omega)=\inf \left\{\int_{\Omega}\Vert X\Vert^2\, d{\rm V}:\,X\in \mathfrak{X}(\Omega)\, {\rm with}\, {\rm div}_g X=-1 \right\}.
    $$
\end{proposition}
\begin{proof}
    Let $E$ be the solution to problem \eqref{eq:unua} and let $X\in \mathfrak{X}(\Omega)$ be smooth vector field such that ${\rm div}_g X=-1$. Applying the divergence Theorem and the Cauchy-Schwartz inequality 
    $$
    \begin{aligned}
        \mathcal{T}(\Omega)=&\int_\Omega E\, d{\rm V}=\int_\Omega -E\, {\rm div}_gX\, d{\rm V}=\int_\Omega{\rm div}_g(-E\, X)\, d{\rm V}+\int_\Omega g(\nabla E, X)\, d{\rm V}\\
        =&-\int_{\partial \Omega} E g(X, \nu)\, d{\rm A}+\int_\Omega g(\nabla E, X)\, d{\rm V} \\
        \leq &\left(\int_\Omega \Vert \nabla E\Vert^2\, d{\rm V}\right)^{\frac{1}{2}}\cdot \left(\int_\Omega \Vert X\Vert^2\, d{\rm V}\right)^{\frac{1}{2}}.
    \end{aligned}
    $$
    Since
    $$
\int_\Omega \Vert \nabla E\Vert^2\, d{\rm V} =\int_\Omega {\rm div}_g(E\nabla E)\, d{\rm V}-\int_\Omega E\Delta E\, d{\rm V}=\mathcal{T}(\Omega),
    $$
    we obtain
    $$
\int_\Omega \Vert X\Vert^2\, d{\rm V}\geq \mathcal{T}(\Omega)=\int_\Omega \Vert \nabla E\Vert^2\, d{\rm V}
    $$
    and the proposition is proved. 
\end{proof}

 \section{Geometric flows and bounds for the torsional rigidity }\label{sec:GF}

 Our objective is to derive the evolution of the torsional rigidity of a domain as the metric tensor of the manifold evolves under a geometric flow. Geometric flows are a fundamental tool in differential geometry, facilitating the analysis of geometric structures through their continuous deformation.
\subsection{Geometric flows}
 These flows can be conceptualized as trajectories, $g(t)$,  within the space of metrics, $\mathcal{M}(M^n)$ defined on a given manifold. The evolution of the metric is typically governed by the partial differential equation    $ \frac{\partial g_t}{\partial t}= F(t)\in T_{g_t}\mathcal{M}(M)$. 

 Since the Riemannian metric determines the geometric properties of the manifold, such as distance, volume, and curvature, any change in the metric under a geometric flow leads to the evolution of these properties. The strength of geometric flows lies in their property to determine the evolution of such quantities.

Given that we will use the two variational characterizations of torsional rigidity described in Proposition \ref{prop:upper_bounds_torsion} and Proposition \ref{prop:lower_bounds_torsion}, it is essential to understand the following time-dependent quantities:
$$
\begin{aligned}
t &\mapsto d\V_{t},\\
t &\mapsto \Vert \nabla^{g_t} u\Vert, \quad u \in C_o^\infty(\Omega),\\
t &\mapsto \Vert X\Vert = \sqrt{g_t(X,X)}, \quad X \in \mathfrak{X}(\Omega),\\
t &\mapsto {\rm div}_{g_t} X, \quad X \in \mathfrak{X}(\Omega).
\end{aligned}
$$
The following proposition addresses these evolutions:

\begin{proposition}\label{prop:general}
    Let $\{g_t\}_{t\in [0,T_{\rm max})}$ be a solution to the geometric flow $\frac{\partial g_t}{\partial t}=f$ on the $n$-dimensional manifold $M$.  Then for any $t\in [0,T_{\rm max})$:
    \begin{enumerate}
        \item $ \frac{\partial (d\V_{t})}{ \partial t }=
        \frac{1}{2}\tr_g(f)d\V_{t}$,
        \item $
\frac{\partial}{\partial t}\Vert \nabla^{g_t}u\Vert^2_{g_t}=-f(\nabla^{g_t}u,\nabla^{g_t} u)$, 
\item $
\frac{\partial}{\partial t}g_t(X,X)=f(X,X)$,  
\item and $  \frac{\partial}{\partial t} {\rm div}_{g_t}X=\frac{1}{2}X(\tr_g(f))$.
\end{enumerate}
\end{proposition}
\begin{proof}We have use the Einstein's notation for repeated indices.
  The evolution of the volume form $d\V_{t}$ can  be computed using local coordinates as  follows:
\begin{equation*}
    \begin{aligned}
        \frac{\partial (d\V_{t})}{ \partial t }=\frac{\partial \sqrt{det(g_t)}}{ \partial t }=& \frac{1}{2\sqrt{det(g_t)}} \frac{\partial det(g_t)}{\partial t}
        =  \frac{1}{2\sqrt{det(g_t)}} \frac{\partial det(g_t)}{\partial g_{ij}}\frac{\partial g_{ij}}{\partial t}\\
        =& \frac{1}{2}\sqrt{det(g_t)} g^{ij} f_{ij}
       = \frac{1}{2}\sqrt{det(g_t)} \tr_g(f).
    \end{aligned}
\end{equation*}
  For the evolution of   $\Vert \nabla^{g_t} u\Vert$ we will have that
  $$
  \begin{aligned}
\frac{\partial}{\partial t}\Vert \nabla^{g_t}u\Vert^2_{g_t}=&\frac{\partial}{\partial t}g\left(\nabla^{g_t}u,\nabla^{g_t}u\right)=\frac{\partial}{\partial t}g\left(g^{im}\frac{\partial u}{\partial x^m}\frac{\partial}{\partial x^i},g^{jl}\frac{\partial u}{\partial x^j}\frac{\partial}{\partial x^l}\right)\\
=&\frac{\partial}{\partial t}\left(g^{im}\frac{\partial u}{\partial x^m}g^{jl}\frac{\partial u}{\partial x^j}g_{il}\right)=\frac{\partial}{\partial t}\left(\delta^{m}_l\frac{\partial u}{\partial x^m}g^{jl}\frac{\partial u}{\partial x^j}\right)\\
=&\frac{\partial u}{\partial x^l}\frac{\partial u}{\partial x^j}\frac{\partial}{\partial t}\left(g^{jl}\right).
  \end{aligned}
$$  
But taking into account that
$
\frac{\partial}{\partial t}\left(g^{rs}g_{sl}\right)=0$
we can easily conclude that 
$\frac{\partial}{\partial t}\left(g^{jl}\right)=-g^{js}f_{sm}g^{ml}$,
and thence
\begin{equation*}
      \begin{aligned}
\frac{\partial}{\partial t}\Vert \nabla^{g_t}u\Vert^2_{g_t}=&-\frac{\partial u}{\partial x^l}\frac{\partial u}{\partial x^j}g^{jm}f_{ms}g^{ls}=- \left( \nabla^{g_t}u\right)^sf_{ms}\left( \nabla^{g_t}u\right)^m\\
=&-f(\nabla^{g_t}u,\nabla^{g_t} u).
  \end{aligned}
\end{equation*}
For the evolution of $\Vert X\Vert^2$ we will just obtain  
$$
\frac{\partial}{\partial t}g_t(X,X)=\frac{\partial}{\partial t}g_{ij}X^iX^j=f_{ij}X^iX^j=f(X,X).
$$
And finally we can compute the evolution of the divergence of a vector field by using the evolution of the volume form as
$$
\begin{aligned}
    \frac{\partial}{\partial t} {\rm div}_{g_t}X=& \frac{\partial}{\partial t} \bigg( \frac{1}{\sqrt{det(g_t)}}\frac{\partial}{\partial x^j}\bigg( X^j\sqrt{det(g_t)} \bigg)\bigg)\\
    =&-\frac{1}{2 \sqrt{det(g_t)}} tr_g(f)\frac{\partial}{\partial x^j}\bigg( X^j\sqrt{det(g_t)} \bigg)\\&+\frac{1}{\sqrt{det(g_t)}}\frac{\partial}{\partial x^j}\bigg(\frac{1}{2} X^j\sqrt{det(g_t)} \tr_g(f)\bigg)
    =\frac{1}{2}X^j\frac{\partial}{\partial x^j}(\tr_g(f))\\
    =&\frac{1}{2}X(\tr_g(f)).
\end{aligned}
$$    
\end{proof}
\subsection{Bounds for the Torsional Rigidity under a Geometric Flow}
With the variational characterization of the torsional rigidity given by propositions \ref{prop:lower_bounds_torsion} and \ref{prop:upper_bounds_torsion}, and with the evolution of the relevant geometric tensors under a geometric flow given by proposition \ref{prop:general} we can provide lower and upper bounds for the torsional rigidity of a domain in terms of its initial torsonial rigidity.
In the following theorem lower bounds are stated.
\begin{theorem}\label{teo:lower}
    Let $\{g_t\}_{t\in [0,T_{\rm max})}$ be a solution to the geometric flow $\frac{\partial g_t}{\partial t}=f$ on the $n$-dimensional manifold $M$. Let $\Omega \subset M$ be a precompact domain with smooth boundary. Suppose that  for any $t\in [0,T_{\rm max})$:
    \begin{enumerate}
        \item There exists a smooth function function $L:\mathbb{R}\to \mathbb{R}$ such that for any point $p\in M$ tangent vector $v\in T_pM$
        $$
f(v,v)\geq L(t) g_t(v,v).
        $$
        \item There exist two smooth functions $l,u:\mathbb{R}\to\mathbb{R}$ such that for any point $p\in M$
        $$
 l(t)\leq \tr_g(f)\leq u(t).
        $$
    \end{enumerate}
    
    Then, for any $t\in [0, T_{\rm max})$, the torsional rigidity $\mathcal{T}(\Omega_t)$ of $\Omega$ with respect to $g_t$ is bounded from below by
  
$$
\begin{aligned}
    \mathcal{T}(\Omega_t)\geq&
    e^{\int_0^t\left(l(s)-\frac{1}{2}u(s)+L(s)\right)ds}\mathcal{T}(\Omega_0).
\end{aligned}
$$

\end{theorem}
\begin{proof}
  Let $E_0$ be the solution for the problem \eqref{eq:unua} for the metric $g_0$ in $\Omega$, namely, let $E_0$ be the initial mean exit time function of $\Omega$ at time $t=0$. Let the following function be defined as,
$$
F(t):=\int_\Omega\Vert \nabla^{g_t} E_0\Vert^2\, d\V_t ,
$$
Then, by using proposition \ref{prop:general},
$$
\begin{aligned}
F'(t)=&\int_\Omega -f(\nabla^{g_t} E_0, \nabla^{g_t} E_0)\, d\V_t+ \frac{1}{2}\int_\Omega\Vert \nabla^{g_t} E_0\Vert^2 \tr_g(f)\, d\V_t\\
\leq &\left(-L(t)+\frac{1}{2} u(t)\right) F(t),
\end{aligned}
$$
which implies 
$$
F(t)\leq e^{\int_0^t\left(-L(s)+\frac{1}{2} u(s)\right)ds}F(0).
$$
Similarly, for the function
$
G(t):=\int_\Omega E_0\, d\V_t
$,
$$
G'(t)=\frac{1}{2}\int_\Omega E_0 \, \tr_g(f)\, d\V_t\geq\frac{1}{2}l(t) G(t),
$$
and thence for any $t\in [0,T_{\rm max})$
$$
G(t)\geq e^{\frac{1}{2}\int_0^tl(s)ds}G(0).
$$
Finally by proposition \ref{prop:upper_bounds_torsion},
$$
\begin{aligned}
    \mathcal{T}(\Omega_t)\geq&\frac{\left(\int_\Omega E^0\, d\V_t\right)^2}{\int_\Omega \Vert \nabla^{g_t}E^0\Vert^2\, d\V_t}= \frac{G^2(t)}{F(t)}\geq e^{\int_0^t\left(l(s)-\frac{1}{2}u(s)+L(s)\right)ds} \frac{G^2(0)}{F(0)}\\
    =&e^{\int_0^t\left(l(s)-\frac{1}{2}u(s)+L(s)\right)ds} \frac{\left(\int_\Omega E^0\, d\V_0\right)^2}{\int_\Omega \Vert \nabla^{g_t}E^0\Vert^2\, d\V_0}\\
    =&e^{\int_0^t\left(l(s)-\frac{1}{2}u(s)+L(s)\right)ds}\mathcal{T}(\Omega_0),
\end{aligned}
$$
and the theorem is proved.
\end{proof}
For upper bounds of the torsional rigidity of a domain in terms of the initial torsional rigidity we can state the following theorem
\begin{theorem}\label{teo:upper}
    Let $\{g_t\}_{t\in [0,T_{\rm max})}$ be a solution to the geometric flow $\frac{\partial g_t}{\partial t}=f$ on the $n$-dimensional manifold $M$. Let $\Omega \subset M$ be a precompact domain with smooth boundary. Suppose that  for any $t\in [0,T_{\rm max})$:
    \begin{enumerate}
        \item There exists a smooth function function $U:\mathbb{R}\to \mathbb{R}$ such that for any point $p\in M$ tangent vector $v\in T_pM$
        $$
f(v,v)\leq U(t) g_t(v,v).
        $$
        \item There exists a smooth function $u:\mathbb{R}\to\mathbb{R}$ such that for any point $p\in M$
        $$
  \tr_g(f)= u(t).
        $$
    \end{enumerate}
    
    Then, for any $t\in [0, T_{\rm max})$, the torsional rigidity $\mathcal{T}(\Omega_t)$ os $\Omega$ with respect to $g_t$ is bounded from above by
  
$$
\begin{aligned}
    \mathcal{T}(\Omega_t)\leq&
e^{\int_0^t\left(U(s)+\frac{1}{2}u(s)\right)ds}\mathcal{T}(\Omega_0).
\end{aligned}
$$
\end{theorem}
\begin{proof}
    Let us set $X=\nabla^{g_0}E^0$ with  $E_0$  the solution for the problem \eqref{eq:unua} for the metric $g_0$ in $\Omega$. Then by proposition \ref{prop:general} 
    $$
\begin{aligned}
    \frac{\partial}{\partial t}{\rm div}_{g_t}X=\frac{1}{2}X(\tr_g(f))=\frac{1}{2}X(u(t))=0.
\end{aligned}
    $$
    This allows us to conclude that the divergence of $X$ is constant throughout the flow and hence
    $$
{\rm div}_{g_t}X={\rm div}_{g_0}X={\rm div}_{g_0}\nabla^{g_0}E^0=\Delta_{g_0}E^0=-1.
    $$
    By using proposition \ref{prop:lower_bounds_torsion} 
    $$
\mathcal{T}(\Omega_t)\leq \int_{\Omega}\Vert X\Vert^2\, d\V_t.
    $$
    We must therefore compute the derivative of the following function
    $$
t\mapsto F(t):=\int_{\Omega}\Vert X\Vert^2\, d\V_t.
    $$
    By using again proposition \ref{prop:general}
    $$
\begin{aligned}
    F'(t)=\int_\Omega f(X,X)\, d\V_t+\frac{1}{2}u(t)\int_\Omega\Vert X\Vert^2\, d\V_t\leq \left(U(t)+\frac{1}{2}u(t)\right)F(t).
\end{aligned}
    $$
    Then 
    $$
    \begin{aligned}
        F(t)\leq &e^{\int_0^t\left(U(s)+\frac{1}{2}u(s)\right)ds}F(0)=
e^{\int_0^t\left(U(s)+\frac{1}{2}u(s)\right)ds}\int_{\Omega}g_0(X,X)\, d\V_0\\
=&
e^{\int_0^t\left(U(s)+\frac{1}{2}u(s)\right)ds}\int_{\Omega}\Vert \nabla^{g_0}E^0\Vert^2\, dV_0=e^{\int_0^t\left(U(s)+\frac{1}{2}u(s)\right)ds}\mathcal{T}(\Omega_0),
    \end{aligned}
    $$
    and the theorem is proved.
\end{proof}
\section{Torsional rigidity under Ricci flow and IMCF}\label{sec:cinc} In this section we will concretize the lower and upper bounds obtained in theorems \ref{teo:lower} and \ref{teo:upper} for the case of Ricci and Inverse Mean Curvature flow
\subsection{Ricci Flow}
The Ricci flow was introduced in Hamilton's in 1982 in the paper ``Three-manifolds with positive Ricci curvature''.
The Ricci flow is a way to deform a Riemannian manifold $(M^n,g)$ using the Ricci tensor in the following way, 
    $$
\frac{\partial}{\partial t}g_t=-2\Ric_{g_t}.
$$
Namely,
\begin{definition}\cite{Andrews}
    We call $(g(t))_{t\in [0,T]}$ the Ricci flow solution for the initial condition $g(0)=g_0 \in \mathcal{M}(M^n)$ if for all $t\in [0,T)$ 
    $$
\frac{\partial}{\partial t}g^t_{ij}=-2{\rm Ric}_{ij}.
$$
\end{definition}
Therefore we can summaries the properties of interest under the evolution of $g_t$ by Ricci flow as follows, 

 \begin{proposition}\label{prop:rf}
    Let $\{g_t\}_{t\in [0,T_{\rm max})}$ be a solution to the geometric flow $\frac{\partial g_t}{\partial t}=f$ on the $n$-dimensional manifold $M$.  Then for any $t\in [0,T_{\rm max})$:
    \begin{enumerate}
        \item $ \frac{\partial (d\V_{t})}{ \partial t }=
        -{\rm Scal}_{g_t}d\V_{t}$,
        \item $
\frac{\partial}{\partial t}\Vert \nabla^{g_t}u\Vert^2_{g_t}=2{\rm Ric}_{g_t}(\nabla^{g_t}u,\nabla^{g_t} u)$, 
\item $
\frac{\partial}{\partial t}g_t(X,X)=-2{\rm Ric}_{g_t}(X,X)$,  
\item and $  \frac{\partial}{\partial t} {\rm div}_{g_t}X=-X({\rm Scal}_{g_t})$.
\end{enumerate}
\end{proposition}
\begin{proof}From the notation of the previous section 
 $$
 f=-2{\rm Ric},\quad \tr_g(f)=-2{\rm Scal}_{g_t}.
 $$
 then, the proposition follows directly from the proposition \ref{prop:general}
\end{proof}
 Thence theorem \ref{teo:lower} can be written as
 \begin{theorem}
Let $(M^n,(g_t)_{t\in [0,T_{\rm max})})$ be a Ricci flow solution. Suppose that for any $t\in [0,T_{\rm max})$, any $p\in M$ and any $v\in T_pM$, there exists a function $B:\mathbb{R}\to\mathbb{R}$ such that 
$$
{\rm Ric}_{g_t}(v,v)\leq B(t)g_t(v,v). 
$$
Suppose moreover that for any $t\in [0,T_{\rm max})$ and any $p\in M$
$$
b(t)\leq {\rm Scal}_{g_t}(p)\leq c(t).
$$
Then for any precompact domain $\Omega\subset M$ with smooth boundary 
$$
\begin{aligned}
\mathcal{T}(\Omega_t)\geq& e^{\int_0^t\left(b(s)-2c(s)-2B(s)\right)ds}\mathcal{T}(\Omega_0). 
\end{aligned}
$$
 \end{theorem}

Furthermore theorem \ref{teo:upper} can be read as
\begin{theorem}Let $(M^n,(g_t)_{t\in [0,T_{\rm max})})$ be a Ricci flow solution. Suppose that for any $t\in [0,T_{\rm max})$, any $p\in M$ and any $v\in T_pM$, there exists a function $A:\mathbb{R}\to\mathbb{R}$ such that 
$$
{\rm Ric}_{g_t}(v,v)\geq A(t)g_t(v,v). 
$$
Suppose moreover that for any $t\in [0,T_{\rm max})$ and any $p\in M$ there exists a function $b:\mathbb{R}\to\mathbb{R}$ such that
$$
 {\rm Scal}_{g_t}(p)=b(t).
$$
Then for any precompact domain $\Omega\subset M$ with smooth boundary 
$$
\begin{aligned}
 \mathcal{T}(\Omega_t)\leq& e^{-\int_0^t\left(2A(s)+b(s)\right)ds}\mathcal{T}(\Omega_0). 
\end{aligned}
$$  
\end{theorem}
Since  by proposition \ref{prop:rf} we know that
$$
\frac{d}{dt}\V(\Omega_t)=\frac{d}{dt}\int_\Omega d\V_t=-\int {\rm Scal}_{g_t}\, d\V_t=-b(t)\V(\Omega_t),
$$
Thence joining the above two theorems we can state the following theorem

\begin{samelettertheoremA}Let $(M^n,(g_t)_{t\in [0,T_{\rm max})})$ be a Ricci flow solution. Suppose that for any $t\in [0,T_{\rm max})$ and any $p\in M$ there exists a function $b:\mathbb{R}\to\mathbb{R}$ such that
$$
 {\rm Scal}_{g_t}(p)=b(t).
$$
Suppose moreover that for any $t\in [0,T_{\rm max})$, any $p\in M$ and any $v\in T_pM$, there exists two functions $A,B:\mathbb{R}\to\mathbb{R}$ such that 
$$
A(t)g_t(v,v)\leq {\rm Ric}_{g_t}(v,v)\leq B(t)g_t(v,v). 
$$
Then for any precompact domain $\Omega\subset M$ with smooth boundary,
$$
\begin{aligned}
 e^{\int_0^t\left(-b(s)-2B(s)\right)ds}\mathcal{T}(\Omega_0)\leq\mathcal{T}(\Omega_t)\leq& e^{\int_0^t\left(-2A(s)-b(s)\right)ds}\mathcal{T}(\Omega_0). 
\end{aligned}
$$
Which implies
$$
\begin{aligned}
 e^{-2\int_0^tB(s)ds}\frac{\mathcal{T}(\Omega_0)}{\V(\Omega_0)}\leq\frac{\mathcal{T}(\Omega_t)}{\V(\Omega_t)}\leq& e^{-2\int_0^tA(s)ds}\frac{\mathcal{T}(\Omega_0)}{\V(\Omega_0)}. 
\end{aligned}
$$
Where here  ${\rm V}(\Omega_t)$ denotes the volume of $\Omega$ with respect to $g_t$.
\end{samelettertheoremA}

Furthermore, if the Ricci tensor is positive (negative) definited  along the flow we can always use $B(t)=b(t)$ ($A(t)=b(t)$) and we obtain the following corollary  
\begin{corollary}\label{cor:Ric}
    Let $(M^n,(g_t)_{t\in [0,T_{\rm max})})$ be a Ricci flow solution. Suppose that for any $t\in [0,T_{\rm max})$ and any $p\in M$
$$
 {\rm Scal}_{g_t}(p)=b(t).
$$
Suppose moreover that there exist $T_0<T_{\rm max}$ such that for any $t\leq T_0$, any $p\in M$ and any $v\in T_pM$
$$
\Ric_{g_t}\leq(\geq) 0.
$$
Then,
$$
t\mapsto \frac{\mathcal{T}(\Omega_t)}{\V(\Omega_t)}
$$
is a non-decreasing (non-increasing) function for $t\in [0,T_0]$, and 
$$
t\mapsto \frac{\mathcal{T}(\Omega_t)}{\V^3(\Omega_t)}
$$
is a non-increasing (non-decreasing) function for $t\in [0,T_0]$.
\end{corollary}
\subsubsection{Einstein metrics}\label{sec:Einstein}
In the case of Einstein metrics we obtained an exact solution for the torsional rigidity, 
        $
    \mathcal{T}(\Omega_t)= \left(1-2\lambda t\right)^{\frac{n+2}{2}} \mathcal{T}(\Omega_0)
    $.
Let $(M^n,g_0)$ be an Einstein manifold, i.e. $\Ric_{g_0}=\lambda\cdot g_0$. Then the Ricci flow solution for the initial condition $g_0$ will be given by the following equation, 
    $$
    g_t= (1-2\lambda t)g_0,
    $$
    for $t\in [0,\infty)$ if $\lambda\leq 0$ or $t\in [0,\frac{1}{2\lambda})$ if $\lambda > 0$.
    Thus the scalar curvature will given by the following function, 
    $$
    {\rm Scal}_{g_t}= \frac{n\lambda}{1-2\lambda \cdot t}.
    $$
    The Ricci curvature satisfies
    $$
\Ric_{g_t}=\frac{\lambda}{1-2\lambda \cdot t}g_t
    $$
    Therefore, the Ricci flow solution will satisfy the conditions of the previous theorem and we can set $A(t)=B(t)=\frac{\lambda}{1-2\lambda \cdot t}$ and $b(t)=\frac{n\lambda}{1-2\lambda \cdot t}$. Then for any precompact domain $\Omega\subset M$ with smooth boundary
    $$
\mathcal{T}(\Omega_t)=e^{\int_0^t\frac{-(n+2)\lambda}{1-2\lambda s}ds}\mathcal{T}(\Omega_0)=\left(1-2\lambda t\right)^{\frac{n+2}{2}}\mathcal{T}(\Omega_0).
    $$
    Since the volume evolves as $\V(\Omega_t)=\left(1-2\lambda t\right)^{\frac{n}{2}}\V(\Omega_0)$ can be stated that
$$
\frac{\mathcal{T}(\Omega_t)}{\V^\frac{n+2}{2}(\Omega_t)}=\frac{\mathcal{T}(\Omega_0)}{\V^\frac{n+2}{2}(\Omega_0)}.
$$
    
\subsubsection{Normalized Ricci Flow}
In the normalized Ricci flow we have that
$$
\frac{\partial}{\partial t}g_t=-2\Ric_{g_t}+\frac{2}{n}\frac{\int_M {\rm Scal}_{g_t}\, d\V_t}{{\rm V}(M_t)}g_t=-2\Ric_{g_t}+\frac{2}{n} \overline{{\rm Scal}_{g_t}}g_t
$$
In the notation of the previous section 
$$
f=-2\Ric_{g_t}+\frac{2}{n} \overline{{\rm Scal}_{g_t}}g_t,\quad \tr_g(f)=2\left(\overline{{\rm Scal}_{g_t}}-{\rm Scal}_{g_t}\right)
$$
Thence \ref{teo:lower} can be written as
 \begin{theorem}
Let $(M^n,(g_t)_{t\in [0,T_{\rm max})})$ be a normalized Ricci flow solution. Suppose that for any $t\in [0,T_{\rm max})$, any $p\in M$ and any $v\in T_pM$, there exists a function $A:\mathbb{R}\to\mathbb{R}$ such that 
$$
{\rm Ric}_{g_t}(v,v)-\frac{1}{n}\overline{{\rm Scal}_{g_t}}g_t(v,v)\leq A(t)g_t(v,v). 
$$
Suppose moreover that for any $t\in [0,T_{\rm max})$ and any $p\in M$
$$
b(t)\leq \overline{{\rm Scal}_{g_t}}-{\rm Scal}_{g_t}(p)\leq c(t).
$$
Then for any precompact domain $\Omega\subset M$ with smooth boundary 
$$
\begin{aligned}
\mathcal{T}(\Omega_t)\geq& e^{\int_0^t\left(2b(s)-c(s)-2A(s)\right)ds}\mathcal{T}(\Omega_0). 
\end{aligned}
$$
 \end{theorem}

Furthermore, joining the above theorem with theorem \ref{teo:upper} we can state the following theorem
\begin{theorem}
Let $(M^n,(g_t)_{t\in [0,T_{\rm max})})$ be a normalized Ricci flow solution. Suppose that for any $t\in [0,T_{\rm max})$ and any $p\in M$ 
$$
{\rm Scal}_{g_t}(p)=u(t).
$$
Suppose moreover that  for any $t\in [0,T_{\rm max})$, any $p\in M$ and any $v\in T_pM$, there exists two functions $A,B:\mathbb{R}\to\mathbb{R}$ such that 
$$
A(t)g_t(v,v)\leq Z_{g_t}(v,v)\leq B(t)g_t(v,v), 
$$
where $Z$ is the trace-free Ricci tensor ($Z:=\Ric-\frac{1}{n}{\rm Scal}\cdot g$). 
Then for any precompact domain $\Omega\subset M$ with smooth boundary 
$$
\begin{aligned}
 e^{-\int_0^t2B(s)ds}\mathcal{T}(\Omega_0)\leq\mathcal{T}(\Omega_t)\leq& e^{-\int_0^t2A(s)ds}\mathcal{T}(\Omega_0). 
\end{aligned}
$$    
\end{theorem}

\subsection{Inverse Mean Curvature Flow}\label{sec:IMCF}
Let $\Sigma\subset \mathbb{R}^{n+1}$ a compact and mean convex hypersurface of the $n+1$ dimensional Euclidean space $\mathbb{R}^{n+1}$, the flow by inverse mean curvature is a family of isometric immersions $\varphi:M\times[0,T_{\rm max})\to \mathbb{R}^{n+1}$ from a $n$-dimensional manifold $M$ to $\mathbb{R}^{n+1}$
such that
\begin{equation}\label{eq:IMCF}  
\left\{
\begin{aligned}
    \varphi(M,0)=&\Sigma\\
    \frac{\partial}{\partial t}\varphi(p,t)=&\frac{-\Vec{H}_t(p)}{\left\Vert \Vec{H}_t(p)\right\Vert^2},
\end{aligned}\right.
\end{equation}
where $\Vec{H}_t(p)$ is the mean curvature of $\varphi(\cdot,t)$ at $p\in M$. Several papers have studied the evolution of the first eignenvalue of the Laplace-Beltrami operator under direct or inverse mean curvature flow, \cite{Ho-Pyo} or \cite{Lambert2016inver-34677} . 

The evolution equations for the inverse mean curvature flow can be founded in \cite{Gerhard1999} for instace. The following equation sumarizes some of the relevan properties

\begin{proposition}\label{prop:volution_equations}
Let $\Sigma\subset \mathbb{R}^{n+1}$ a compact and mean convex hypersurface of the $n+1$ dimensional Euclidean space $\mathbb{R}^{n+1}$. Then the equations of the evolution of the metric, volume form and mean curvature under the inverse mean curvature flow will be as follows,
       \begin{equation}
        \left\lbrace\begin{aligned}
            &\frac{\partial}{\partial t}g_{ij}=2\frac{h_{ij}}{H}\\
            &\frac{\partial}{\partial t}d{\rm V}_t=d{\rm V}_t\\
            &\frac{\partial}{\partial t}H=-\Delta \frac{1}{H}
        \end{aligned}\right.
    \end{equation} 
\end{proposition}
In the notation of the previous section 
$$
f=2\frac{h_{ij}}{H},\quad \tr_g(f)=2.
$$
As we previously mentioned the time that the surface remains strictly convex can be finite and it will be denoted by $T_{mc}$. This maximum time of mean convexity has been studied by several authors in \cite{maxtime}, \cite{Lambert2016inver-34677} or in \cite{Urbas}. For this maximum time we can state the following bound for any $p\in M$ and any $v\in T_pM$
$$
0<\frac{h(v,v)}{H}\leq 1
$$
And in this case theorem \ref{teo:lower} and \ref{teo:upper} can be written together as
\begin{theorem}
Let M be an n-dimensional manifold and let $\Omega\subset M$ be an open and precompact domain of $M$ with smooth boundary $\partial \Omega\neq \emptyset$. Then let
    $\varphi:M\times[0,T_{\rm mc })\to \mathbb{R}^{n+1}$ by an inverse mean curvature flow from $M$ to $\mathbb{R}^{n+1}$ such that for all $t\in T_{mc}$ we have that $\varphi_t(M)=\Sigma$ is strictly  convex hypersurface of $\mathbb{R}^{n+1}$.     
    Then
 $$
 e^{t}\mathcal{T}(\Omega_0)\leq \mathcal{T}(\Omega_t)\leq e^{3t}\mathcal{T}(\Omega_0).
 $$   
\end{theorem}
Observe that in this case by using proposition \ref{prop:general}
$$
    \frac{\partial}{\partial t}\V(\Omega_t) =\frac{\partial}{\partial t} \int_{\Omega} d\V_t = \V(\Omega_t).
    $$
which implies that
$$
\V(\Omega_t)=e^t\V(\Omega_0).
$$
This allows us to state the following
\begin{samelettertheoremB}
   Let M be an n-dimensional manifold and let $\Omega\subset M$ be an open and precompact domain of $M$ with smooth boundary $\partial \Omega\neq \emptyset$. Then let
    $\varphi:M\times[0,T_{\rm mc })\to \mathbb{R}^{n+1}$ by an inverse mean curvature flow from $M$ to $\mathbb{R}^{n+1}$ such that for all $t\in T_{mc}$ we have that $\varphi_t(M)=\Sigma$ is strictly  convex hypersurface of $\mathbb{R}^{n+1}$.Thence the function
    $$
 t\mapsto ({\rm V}(\Omega_t))^{-3}\cdot\mathcal{T}(\Omega_t)
    $$
    is non-increasing for $t\in [0,T_{\rm max})$, and the function
    $$
 t\mapsto ({\rm V}(\Omega_t))^{-1}\cdot\mathcal{T}(\Omega_t)
    $$
    is non-decreasing for $t\in [0,T_{\rm max})$. Where here  ${\rm V}(\Omega_t)$ denotes the volume of $\Omega$ with respect to $g_t$.
\end{samelettertheoremB}

\section{Examples of application}\label{sec:examples}
The applications of the theorems \ref{Thm:Ric} and \ref{teo:unua} are extensive. For example, in the case of the evolution of the torsional rigidity under the Ricci flow it can be applied to Homogeneous spaces, as it has been done for homogeneous spheres, and for Einstein metrics. Additionally, the applications of these results to the inverse mean curvature flow (IMCF) are also broad. This article computes the case of a strictly convex, free-boundary, disk-type hypersurface within the ball.

\subsection{Examples of the evolution of the torsional rigidity under Ricci Flow}

Finally we will give bounds in the torsional rigidity for two of the most studied cases in Ricci flow, the Heisenberg group and in homogeneous spheres. 
\subsubsection{Homogeneous metrics in ${\rm Nil}_3$ under Ricci flow}\label{Heisenberg}
Recall that the  Heisenberg group is algebraically the matrix lies group of matrices of the form
$$
{\rm Nil}_3\simeq \left\lbrace\begin{pmatrix}1& x& z \\
0&1&y\\
0&0&1\end{pmatrix}\, :\, (x,y,z)\in \mathbb{R}^3 \right\rbrace.
$$
Topologically ${\rm Nil}_3$ is diffeomorphic to $\mathbb{R}^3$. If we endow ${\rm Nil}_3$ with a left invariant metric, the Ricci flow of the metric tensor is always a left invariant metric $g(t)$ that can be written in a Milnor farme $\{F_i\}$ as 
$$
g_t=A(t)\theta^1\otimes\theta^1+B(t)\theta^2\otimes\theta^2+C(t)\theta^3\otimes\theta^3.
$$
where $\{\theta^i\}$ are the duals formts to $\{F_i\}$, $\theta^i(F_j)=\delta^i_j$ and $A,B,C$ are positive functions with analytical expression given by, see \cite{ChowLiNi}:
$$
\left\lbrace\begin{aligned}
    A(t)=&A_0^{\frac{2}{3}}B_0^{\frac{1}{3}}C_0^\frac{1}{3}\left(12t+\frac{B_0C_0}{A_0}\right)^{-\frac{1}{3}}\\
    B(t)=&A_0^{\frac{1}{3}}B_0^{\frac{2}{3}}C_0^{-\frac{1}{3}}\left(12t+\frac{B_0C_0}{A_0}\right)^{\frac{1}{3}}\\
    C(t)=&A_0^{\frac{1}{3}}B_0^{-\frac{1}{3}}C_0^\frac{2}{3}\left(12t+\frac{B_0C_0}{A_0}\right)^{\frac{1}{3}}
\end{aligned}\right.
$$
The Ricci tensor is given by
$$
\Ric_{g_t}=\frac{2A^2}{BC}\theta^1\otimes\theta^1-\frac{2A}{C}\theta^2\otimes\theta^2-\frac{2A}{B}\theta^3\otimes\theta^3.
$$
Then for any $p\in {\rm Nil}_3$ and any $v\in T_p{\rm Nil}_3$ ,$v\neq 0$,
$$
-\frac{2}{12t+\frac{B_0C_0}{A_0}}=-\frac{2A}{BC}\leq \frac{\Ric_{g_t}(v,v)}{g_t(v,v)}\leq \frac{2A}{BC}=\frac{2}{12t+\frac{B_0C_0}{A_0}}
$$
Moreover the scalar curvature is given by
$$
{\rm Scal}_{g_t}=-\frac{2A}{BC}=-\frac{2}{12t+\frac{B_0C_0}{A_0}}.
$$
Hence by a straightforward application of Theorem  \ref{Thm:Ric}, we state the following theorem
\begin{theorem}
Let $g_0$ be a Left invariant metric tensor on ${\rm Nil}_3$ given with respect to the Milnor farme $\{F_i\}$ as
$$
g_0=A_0\theta^1\otimes\theta^1+B_0\theta^2\otimes\theta^2+C_0\theta^3\otimes\theta^3.
$$
Let $g_t$ be the solution for the Ricci flow for $t\in [0,\infty)$ with initial condition $g_0$. Then for any precompact domain $\Omega\subset {\rm Nil}_3$ with smooth boundary
$$
\left(1+\frac{12A_0}{B_0C_0}t\right)^{-\frac{1}{6}}\mathcal{T}(\Omega_0)\leq\mathcal{T}(\Omega_t)\leq \left(1+\frac{12A_0}{B_0C_0}t\right)^{\frac{1}{2}}\mathcal{T}(\Omega_0).
$$
\end{theorem}
Taking into account that by using proposition \ref{prop:rf},
$$
\V(\Omega_t)=\left(1+\frac{12A_0}{B_0C_0}t\right)^\frac{1}{6}\V(\Omega_0)
$$
We can state the following corollary
\begin{corollary}
    Let $g_0$ be a Left invariant metric tensor on ${\rm Nil}_3$ given with respect to the Milnor farme $\{F_i\}$ as
$$
g_0=A_0\theta^1\otimes\theta^1+B_0\theta^2\otimes\theta^2+C_0\theta^3\otimes\theta^3.
$$
Let $g_t$ be the solution for the Ricci flow for $t\in [0,\infty)$ with initial condition $g_0$. Then for any precompact domain $\Omega\subset {\rm Nil}_3$ with smooth boundary
\begin{enumerate}
    \item The function 
    $$
t\mapsto \V(\Omega_t)\mathcal{T}(\Omega_t),
    $$
    is non decreasing on $t\in [0,\infty)$,
    \item  and the function
    $$
t\mapsto \frac{\mathcal{T}(\Omega_t)}{\V(\Omega_t)^3},
    $$
    is non increasing on $t\in [0,\infty)$.
\end{enumerate}
\end{corollary}

\subsubsection{Homogeneous metrics in $SU(2)$  under Ricci flow}\label{sec:su2}

Given  the Milnor frame, left invariant vector fields, $\{F_1,F_2,F_3\}$   of $SU(2)$, and let $\{\theta^1,\theta^2, \theta^3\}$ the associated one forms. In this subsection it will be consider the evolution of the left-invariant metric tensor 
$$
g=\epsilon A_0 \theta^1\otimes\theta^1+B_0 \theta^2\otimes\theta^2+C_0\theta^3\otimes\theta^3.
$$
The evolution under Ricci flow is given, see \cite{ChowLiNi}, by
$$
g_t=D(t) \theta^1\otimes\theta^1+B(t) \theta^2\otimes\theta^2+C(t)\theta^3\otimes\theta^3,
$$
with  $D(0)=\epsilon A_0$, $B(0)=B_0$ and $C(0)=C_0$ and
$$
\begin{aligned}
    \frac{d}{dt}B&= -8+4\frac{C^2+D^2-B^2}{CD}\\
    \frac{d}{dt}C&= -8+4\frac{B^2+D^2-C^2}{BD}\\
    \frac{d}{dt}D&= -8+4\frac{B^2+C^2-D^2}{BC}
\end{aligned}
$$
Moreover, the orthonormal frame $\displaystyle\left\{X_2=\frac{F_2}{\sqrt{B}},X_3=\frac{F_3}{\sqrt{C}}, X_4=\frac{F_1}{\sqrt{D}} \right\}$ diagonalizes the Ricci tensor with
$$
\begin{aligned}
    {\rm Ric}(X_2,X_2)=&\frac{2}{BCD}\left[B^2-(D-C)^2\right]\\
    {\rm Ric}(X_3,X_3)=&\frac{2}{BCD}\left[C^2-(D-B)^2\right]\\
    {\rm Ric}(X_4,X_4)=&\frac{2}{BCD}\left[D^2-(B-C)^2\right]
\end{aligned}
$$
 Assuming  $D_0:=\epsilon A_0\leq C_0\leq B_0$ it is obtained that $D\leq C\leq B$, then
 \begin{equation}\label{ode:1}
\frac{d B}{dt}\leq -8+4\frac{D}{C}\leq -4,     
 \end{equation}
which implies a finite time of existence of the flow and 
 $$
0<B-D\leq \left(\frac{B_0-D_0}{D_0}\right)D
 $$
 and hence $B-D\rightarrow 0$ in the evolution of the Ricci flow.  In the case when $\delta:=\frac{B_0-D_0}{D_0}<1$, can be proved that
 $$
0<  {\rm Ric}(X_4,X_4)\leq   {\rm Ric}(X_3,X_3)\leq   {\rm Ric}(X_2,X_2).
 $$
 \begin{theorem}
    Let $SU(2)$ be endowed with the left-invariant metric 
    $$
g=\epsilon A_0 \Theta^1\otimes\Theta^1+B_0 \Theta^2\otimes\Theta^2+C_0\Theta^3\otimes\Theta^3.
$$
Let $g_t$ be the evolution of $g$ under the Ricci flow. Let $\Omega\subset SU(2)$ be a precompact domain with smooth boundary. Suppose that 
$$
\epsilon A_0\leq C_0\leq B_0,\quad \delta:=\frac{B_0-\epsilon A_0}{\epsilon A_0}<1.
$$
Then,
$$
\left[1-\frac{4(1+6\delta)t}{B_0}\right]^{\frac{5(1+\delta)^3}{2(1+6\delta)}}\mathcal{T}(\Omega_0)\leq \mathcal{T}(\Omega_t)\leq \left[1-\frac{4t}{B_0}\right]^{\frac{5(1-\delta)}{2(1+\delta)}}\mathcal{T}(\Omega_0)
$$
For any $t<\min\left\{\frac{B_0}{4(1+6\delta)},T_{\rm max}\right\}$.
\end{theorem}
\begin{proof}
Since homogeneous spheres under this conditions has a monotonicity in the Ricci curvature,
 $$
0<  {\rm Ric}(X_4,X_4)\leq   {\rm Ric}(X_3,X_3)\leq   {\rm Ric}(X_2,X_2).
 $$
In terms of Theorem \ref{Thm:Ric}, where we are using $\overline{A}$ and $\overline{B}$ to avoid confusions with the parameters $A,B,C$ of the metric tensor,  
 $$
\overline{B}(t)={\rm Ric}(X_2,X_2)\quad {\rm and} \quad \overline{A}(t)={\rm Ric}(X_4,X_4).
 $$
 Therefore, for the case of the lower curvature bound,
 $$
 \begin{aligned}
     -b(t)-2\overline{A}(t)=&-{\rm Ric}(X_2,X_2)-{\rm Ric}(X_3,X_3)-3{\rm Ric}(X_4,X_4)\leq -5{\rm Ric}(X_4,X_4)\\
     =& -\frac{10}{BCD}\left[D^2-(B-C)^2\right]\leq -\frac{10}{B^3}\left[D^2-(B-D)^2\right]\\
     =&-\frac{10D^2}{B^3}\left[1-\left(\frac{B-D}{D}\right)^2\right]\leq -\frac{10}{B(1+\delta)^2}(1-\delta^2)
 \end{aligned}
 $$
 by equation \eqref{ode:1},
 $$
 -b(t)-2\overline
{A}(t)\leq -\frac{10(1-\delta^2)}{(1+\delta)^2}\frac{1}{B_0-4t}=\frac{5(1-\delta)}{2(1+\delta)}\frac{d}{dt}\ln\left(B_0-4t\right).
 $$
 Thence, by Theorem \ref{Thm:Ric}
$$
\begin{aligned}
\mathcal{T}(\Omega_t)\leq& e^{\int_0^t\left(-2 \overline{A}(s)-b(s)\right)ds}\mathcal{T}(\Omega_0)\leq  \left[1-\frac{4t}{B_0}\right]^{\frac{5(1-\delta)}{2(1+\delta)}}\mathcal{T}(\Omega_0).
\end{aligned}
$$
 Similarly in the case of the upper bound in the Ricci curvature,
 $$
 \begin{aligned}
     -b(t)-2\overline{B}(t)=&-3{\rm Ric}(X_2,X_2)-{\rm Ric}(X_3,X_3)-{\rm Ric}(X_4,X_4)\geq -5{\rm Ric}(X_2,X_2)\\
     =& -\frac{10}{BCD}\left[B^2-(D-C)^2\right]\geq -\frac{10B^2}{D^3}\\
     \geq&-\frac{10(1+\delta)^3}{B}.
 \end{aligned}
 $$
But
$$
\begin{aligned}
    \frac{d}{dt}B=& -8+4\frac{C^2+D^2-B^2}{CD}=-8+4\frac{C}{D}-4\frac{(B-D)(D+B)}{CD}\\
    \geq&-4-4\frac{\delta(D+B)}{C} \geq -4-8\frac{\delta B}{C}\geq-4-8\frac{\delta B}{D}  \geq-4-8\delta(1+\delta)\\
    =&-4(1+2\delta(1+\delta)) \geq -4(1+6\delta) 
\end{aligned}
$$
which implies
$$
B\geq B_0-4(1+6\delta)t
$$
and
 $$
 \begin{aligned}
     -b(t)-2\overline{B}(t)
     \geq&-\frac{10(1+\delta)^3}{B_0-4(1+6\delta)t}=\frac{5(1+\delta)^3}{2(1+6\delta)}\frac{d}{dt}\ln\left(B_0-4(1+6\delta)t\right)
 \end{aligned}
 $$
 Thence, by using Theorem \ref{Thm:Ric}
$$
\begin{aligned}
\mathcal{T}(\Omega_t)\geq& e^{\int_0^t\left(-2 \overline{B}(s)-b(s)\right)ds}\mathcal{T}(\Omega_0)\geq  \left[1-\frac{4(1+6\delta)t}{B_0}\right]^{\frac{5(1+\delta)^3}{2(1+6\delta)}}\mathcal{T}(\Omega_0).  \mbox{\qedhere}
\end{aligned}
$$

\end{proof}
\begin{remark} In the case $A_0=B_0=C_0=1$ and $0<\epsilon<1$ we obtain the family of Berger spheres.
    Observe that for the case $\delta=0$, the initial metric tensor is the round sphere 
    $$
g=B_0 \theta^1\otimes\theta^1+B_0 \theta^2\otimes\theta^2+B_0\theta^3\otimes\theta^3
    $$
    and We conclude
    $$
\mathcal{T}(\Omega_t)= \left[1-\frac{4t}{B_0}\right]^{\frac{5}{2}}\mathcal{T}(\Omega_0).
    $$
    for any $t\in \left[0,T_{\rm max}=\frac{B_0}{4}\right)$. Observe that in this case $SU(2)$ is a Einstein manifold with $\lambda=\frac{2}{B_0}$. Moreover, 
\end{remark}

\subsection{Example the evolution of the torsional rigidity under IMCF} 
In this subsection we will apply Theorem \ref{teo:unua} to the case of strictly convex free boundary disk-type hypersurface in a ball. The definition of such a hypersurfaces is as follows
\begin{definition}\label{def:fb}
An strictly convex hypersurface with the topology of a disk and satisfying the following equations, 
   \begin{enumerate}[(a)]
       \item $X_0(\partial \mathbb{D})\subset \mathbb{S}^n,$
       \item $\langle \dot\gamma(0),\widetilde{N}\rangle\geq 0,\quad \forall \gamma\in C^1((-\epsilon,0],X_0(\mathbb{D})):\gamma(0)\in \partial X_0(\mathbb{D})$
       \item $\langle N_0\vert_{\partial \mathbb{D}},\widetilde{N}\vert_{\partial \mathbb{D}}\rangle=0$
   \end{enumerate}

Is known as an \emph{strictly convex free boundary disk-type hypersurface in the ball $\mathbb{B}\subset \mathbb{R}^{n+1}$}. An example of this kind of hypersurfaces can be the embedded disk in the center of a sphere. 
\end{definition}

Lambert and Scheuer proved in \cite{Lambert2016inver-34677} that given an embedding $X_0: \mathbb{D}\to \mathbb{R}^{n+1}$ of the $n$ dimensional disk  $\mathbb{D}$ as a strictly convex free boundary disk-type hypersurface in the ball. There exists a family of embeddings  $X_t:\mathbb{D}\to \mathbb{R}^{n+1}$ satisfying IMCF and a finite time $T_{\rm mc}<\infty$ such that the embeddings $X_t$ converge to the embedding of a flat unit disk as $t\to T_{\rm mc}$, as Lambert and Scheuer showed.

Following the ideas of Pak Tung Ho and Juncheol Pyo in \cite{Ho-Pyo} and by using the variational characterization of the the torsional rigidity given in \eqref{eq:torsionalvariation} as a corollary of Theorem \ref{teo:unua} we can state.
\begin{corollary}Let $\Sigma$ be a strictly convex free boundary disk-type hypersurface  in the ball $\mathbb{B}\subset \mathbb{R}^{n+1}$, $\mathcal{T}(\Sigma)$ the torsional rigidity of $\Sigma$ and $\V(\Sigma)$ the volume of $\Sigma$. Then we have
$$
\frac{\mathcal{T}(\Sigma)}{\V^3(\Sigma)}\geq \frac{\mathcal{T}(\mathbb{D})}{\V^3(\mathbb{D})},\quad \frac{\mathcal{T}(\Sigma)}{\V(\Sigma)}\leq \frac{\mathcal{T}(\mathbb{D})}{\V(\mathbb{D})}
$$
where $\mathcal{T}(\mathbb{D})$ and $\V(\mathbb{D})$ are respectively the torsional rigidity and the volume of a unit disk $\mathbb{D}$ passing through the center of the unit ball $\mathbb{B}\subset \mathbb{R}^{n+1}$.
\end{corollary}
\begin{remark}
    The mean exit time function of a unit disk, as it was computed in the example \ref{ex1} , is given by
    $$
x\mapsto E(x)=\frac{1}{2n}\left(1-\vert x\vert ^2\right)
    $$
then, the torsional rigidity is given by
$$
\mathcal{T}(\mathbb{D})=\frac{\omega_{n-1}}{2n}\int_0^1\left(1-r ^2\right)r^{n-1}dr=\frac{\omega_{n-1}}{n^2(2+n)},
$$
where $\omega_{n-1}$ is the volume of the unit $n-1$-dimensional sphere. Moreover
$$
\V(\mathbb{D})=\omega_{n-1}\int_0^1r^{n-1}dr=\frac{\omega_{n-1}}{n},
$$
Thence
$$
\frac{n}{\omega_{n-1}^2(n+2)}\V^3(\Sigma)\leq \mathcal{T}(\Sigma)\leq \frac{1}{n(n+2)}\V(\Sigma)
$$
and we can state
\begin{corollary}
    Let $\Sigma$ be a strictly convex free boundary disk-type hypersurface  in the ball $\mathbb{B}\subset \mathbb{R}^{n+1}$. Then
    $$
\V(\Sigma)\leq \V(\mathbb{D}),
    $$
    and
    $$
\mathcal{T}(\Sigma)\leq \mathcal{T}(\mathbb{D}).
    $$
\end{corollary}
\end{remark}
\begin{remark}
    For every domain $\Omega$ the functional $\Omega\mapsto \frac{\mathcal{T}(\Omega)}{\V(\Omega)^\frac{n+2}{2}}$ is scale invariant. 
It is proved that for any domain of $\mathbb{R}^n$,
$$
\frac{\mathcal{T}(\Omega)}{\V(\Omega)^\frac{n+2}{2}}\leq \frac{\mathcal{T}(\mathbb{D})}{\V(\mathbb{D})^\frac{n+2}{2}}.
$$
Observe that for our results we can state that for strictly convex free boundary disk-type hypersurfaces  in the ball $\mathbb{B}\subset \mathbb{R}^{n+1}$
$$
\frac{\mathcal{T}(\Omega)}{\V(\Omega)^\frac{n+2}{2}}\geq \frac{\mathcal{T}(\mathbb{D})}{\V(\mathbb{D})^\frac{n+2}{2}}\cdot \left(\frac{\V(\Sigma)}{\V(\mathbb{D})}\right)^2
$$
\end{remark}

\end{document}